\documentclass[11pt,reqno]{amsart}

\usepackage{amsfonts}
\usepackage{amsmath}
\usepackage{amssymb, esint}
\usepackage{amsthm,color}
\usepackage{enumerate}
\usepackage{mathtools}
\usepackage{enumitem}
\usepackage{url}
\usepackage[hidelinks, bookmarksdepth=3]{hyperref}
\usepackage{xcolor}
\usepackage{verbatim}
\usepackage{cancel}
\usepackage{ulem}
\usepackage{caption,subcaption}
\usepackage{tikz}
\usepackage{tikz-3dplot}
\usepackage{longtable}
\mathtoolsset{showonlyrefs}

\usepackage{soul}

\tdplotsetmaincoords{70}{110}

\usepackage{geometry}
\newcommand{\R}{{\mathbb R}}

\newcommand{\N}{{\mathbb N}}

\def\0{{\mathbf 0}}
\def\loc{{\textup{loc}}}

\def\p{{\partial}}

\newcommand{\eps}{\varepsilon}

\newcommand{\supp}{\operatorname{supp}}

\newcommand{\diam}{\operatorname{diam}}
\newcommand{\dist}{\operatorname{dist}}

\newcommand{\tr}{\operatorname{tr}}

\newcommand\norm[1]{\left\Arrowvert {#1} \right\Arrowvert}
\newcommand\abs[1]{\left\arrowvert {#1} \right\arrowvert}

\theoremstyle{plain}
\newtheorem{thm}{Theorem}[section]

\newtheorem{cor}[thm]{Corollary}
\newtheorem{lem}[thm]{Lemma}
\newtheorem{prop}[thm]{Proposition}
\newtheorem{definition}[thm]{Definition}

\newcommand{\thistheoremnames}{}
\newtheorem*{genericthms}{\thistheoremnames}
\newenvironment{para*}[1]
  {\renewcommand{\thistheoremnames}{#1}%
   \begin{genericthms}}
  {\end{genericthms}}

\theoremstyle{remark}

\newtheorem*{claim*}{Claim}

\numberwithin{equation}{section}

\title[ACF Almost Monotonicity at Infinity]
{ACF Almost Monotonicity at Infinity
with Applications to Perturbed Global Solutions}

\author{Simon Eberle}
\address[Simon Eberle]{}
\email{simon.eberle.math@gmail.com}

\author{Anthony Salib}
\address[Anthony Salib]{Department of Mathematics, ETH Zürich, Rämistrasse 101, 8092 Zürich, Switzerland.}
\email{anthony.salib@math.ethz.ch}

\author{Georg S. Weiss}
\address[Georg S. Weiss]{Department of Mathematics, University of Duisburg-Essen, Thea-Leymann-Strasse 9, 45127 Essen, Germany.}
\email{georg.weiss@uni-due.de}

\author{Henrik Shahgholian}
\address[Henrik Shahgholian]{Department of Mathematics, KTH Royal Institute of Technology, SE-10044 Stockholm, Sweden  
} 
\email{henriksh@kth.se}


\date{\today}

\begin{document}
\begin{abstract}
We study the large-scale behavior of the coincidence set of perturbations of global solutions to the classical obstacle problem in $\mathbb{R}^n\setminus B_1$, with blow-down invariant in the $e_n$ direction. In dimensions $n\geq 3$, we prove that, locally around regular points sufficiently far out, the cross-sections of $\{u=0\}$ perpendicular to $e_n$ are $C^2$ perturbations of ellipsoids. The main ingredient is a new large-scale almost monotonicity formula for the Alt--Caffarelli--Friedman functional. In contrast with the classical small-scale perturbative theory, our argument exploits the stability of the obstacle problem together with the fact that local perturbations vanish under blow-down. The method provides a model mechanism for controlling errors at infinity in stable free boundary problems.
\end{abstract}

\maketitle
\setcounter{tocdepth}{1}
\tableofcontents

\section{Introduction}

Monotonicity formulae are a central tool in the analysis of free boundary problems. In many situations one does not have an exact monotonicity formula, but rather an almost monotonicity formula adapted to perturbations, lower-order terms, or non-homogeneous right-hand sides; see, for instance, \cite{caffarelli2002some,franceschini2021c,salib2025classification}. The general strategy behind such results is often perturbative: one differentiates the relevant functional, identifies the errors produced by the perturbation, and proves that these errors are controlled in the regime under consideration.

For the Alt--Caffarelli--Friedman monotonicity formula (ACF) introduced in \cite{acf84}, this program was carried out at small scales in the celebrated work \cite{caffarelli2002some}. A key feature of that argument is that the functions under consideration vanish at the base point with an algebraic rate. This small-scale vanishing makes the perturbative errors integrable and allows one to recover an almost monotonicity formula (see \cite[Theorem 1.6]{caffarelli2002some}).

The purpose of the present paper is to prove an analogue of this phenomenon at large scales for perturbations of global solutions to the obstacle problem and thus deduce geometric information on the coincidence set of such solutions. The mechanism is fundamentally different from that of \cite{caffarelli2002some}. At infinity one cannot rely on vanishing at a point: indeed, the relevant blow-downs have quadratic growth. Instead, the basic observation is that a local perturbation becomes negligible after rescaling to large scales. Thus the perturbation is not controlled by small-scale vanishing of the solution, but by the shrinking of the perturbed region under blow-down.

As a model problem we consider the classical obstacle problem
\begin{equation}\label{eq:main}
    \Delta u = \chi_{\{u>0\}} \quad \text{in } D,
    \qquad
    u = g \quad \text{on } \partial D,
\end{equation}
where $D\subset \mathbb{R}^n$, $n\geq 2$, and $g$ is a prescribed smooth boundary datum. When $D = \R^n$, solutions of this problem are called global solutions and have been characterized in all dimensions. More specifically, it has been shown, see \cite{eberleduke2022,eberle2022complete}, that the coincidence sets $\{ u = 0\}$ are generalized ellipsoids, in the sense that there exists a sequence of ellipsoids $E_j$ such that $\{ u = 0\} = \lim_j E_j$ in the Hausdorff sense.

The question we raise in this article is what happens if we perturb a global solution to the obstacle problem locally. More specifically, we are interested in the behavior of solutions to the above problem at infinity when $D = \mathbb{R}^n \setminus B_1$. To provide some perspective for the reader, one may consider any global solution in $\mathbb{R}^n$---such as those where the coincidence set is an ellipsoid, a paraboloid, a half-space, or a cylinder with an elliptic or parabolic base---and consider a local perturbation within $B_1$ while preserving the behavior at infinity of the original solution. Such a construction is easily done by fixing any Dirichlet data on $\partial B_1$, and solving the obstacle problem in $B_R\setminus B_1$, with boundary values on $\partial B_R$ as one of the above-mentioned global solutions. It then follows by a standard compactness argument that, along a subsequence, the solutions $u_R$ converge to a solution of our problem (see Appendix \ref{app:A}). 

At such a general level, the question may have several answers, and a classification of all perturbed global solutions seems far from being understood. For instance, an important feature of global solutions is that they are convex and so if $\{u=0\}$ has nonempty interior then the entire free boundary is made up of regular points. Recall that a point $x_0 \in \p \{u>0\}$ is called a regular point, and we write $x_0 \in \operatorname{Reg}(u)$, if the free boundary is smooth in some neighborhood of $x_0$ (see \cite{caffarelli1998obstacle}). If $x_0 \notin \operatorname{Reg}(u)$ then $x_0$ is called a singular point and we write $x_0 \in \operatorname{Sing}(u)$. Moreover, it is shown in \cite{caffarelli1998obstacle} that the entire free boundary is decomposed as
\begin{equation}
    \p \{u>0\} = \operatorname{Reg}(u) \cup \operatorname{Sing}(u).
\end{equation}

Hence, while for global solutions with nonempty interior there holds that $\operatorname{Sing}(u) = \emptyset$, for the perturbed problem this is not necessarily the case and there may be many connected components containing singular points, and such components cannot be ruled out in general. Since the set of singular points can exhibit wild behavior, see the examples constructed in \cite{schaeffer1977some}, one cannot hope to give such a structural theorem as Theorem \ref{theorem:main} on the set of singular points. Nevertheless we are still able to provide some structure for the regular part of the free boundary of perturbed solutions at large scales (see Theorem \ref{theorem:main} below for the precise statement). 

\subsection{Results}
Our first main result is an almost monotonicity formula for the Alt--Caffarelli--Friedman functional at infinity. Recall that, for a function $v$, the ACF functional is
\begin{equation}
\Phi(v,x_0,r) = \frac{1}{r^4}\left(\int_{B_r(x_0)} \frac{|\nabla v_+|^2}{|x-x_0|^{n-2}} \right) \left(\int_{B_r(x_0)}\frac{|\nabla v_-|^2}{|x-x_0|^{n-2}}\right),
\end{equation}
and given a direction $e \in \p B_1$ we define the finite difference quotients
\begin{equation}
    v_{e,h}(x) = \frac{v(x+he) - v(x)}{h}.
\end{equation}

\begin{thm}\label{thm:acf}
Suppose that $w \in C^{1,1}(\R^n)$ is a non-negative solution of $\Delta w = f$ for some $f \in L^{\infty}(\R^n)$ with $\supp(f-\chi_{\{w>0\}}) \subset B_4$. Then, for every $e \in \p B_1$, for $n\geq 3$, given $\delta >0$ and $h \in (0,1)$,  there exists a constant $C = C(\delta, h, \norm{D^2w}_{L^{\infty}(\R^n)},\norm{f}_{L^{\infty}},n)$ such that
\begin{equation}
\Phi(w_{e,h}, x_0, r) \leq  \Phi(w_{e,h}, x_0, R) + C r^{-1/2} + \delta,
\end{equation}
for each $x_0 \in \p\{w>0\}$ and $1<r<R$. For $n=2$, given $h \in (0,1)$ there exists a constant $C = C(h, \norm{D^2w}_{L^{\infty}(\R^n)},\norm{f}_{L^{\infty}})$, such that 
\begin{equation}
    \Phi(w_{e,h}, x_0, r) \leq  \Phi(w_{e,h}, x_0, R) + C \log(4R) r^{-1/2} ,
\end{equation}
for each $x_0 \in \p\{w>0\}$ and $1 < r < R$. 
\end{thm}

Notice that Theorem \ref{thm:acf} applies to solutions of \eqref{eq:main} after applying a cut-off around $B_1$ (see Section \ref{section:1} below). Thus, although the ACF functional need not be monotone for the perturbed solution itself, the  defect of monotonicity tends to zero at infinity. In dimension $n=2$ the same argument gives an almost monotonicity formula with a constant depending on the $\log$ of the outer radius and thanks to this logarithm it is still sufficient for the classification of blow-down limits; see Lemma~\ref{lem:classification} below.  

The proof of Theorem~\ref{thm:acf} is the main new ingredient of the paper, and as mentioned, it is not obtained by differentiating the ACF functional and estimating the resulting error terms. Instead, for each large scale we compare the exterior solution, after a suitable cut-off, with its obstacle replacement in a large ball with the same boundary data. The ACF functional associated with this replacement is genuinely monotone. Stability estimates for the obstacle problem then show that the replacement is sufficiently close to the original cut-off solution. Since the perturbation is confined to a fixed compact set, its effect shrinks at large scales, and one can transfer enough monotonicity from the replacement back to the perturbed solution.

This large-scale mechanism is the essential difference from \cite{caffarelli2002some}: there the perturbative errors are controlled by the algebraic vanishing of the solution at the base point, whereas here they are controlled by the disappearance of the compact perturbation under blow-down.

Theorem \ref{thm:acf} allows us then to classify blow-down limits along moving centers in every dimension $n\geq 2$. When $n\geq 3$ and the blow-down of $u$ is a quadratic polynomial independent of $x_n$, this classification yields the following structure theorem for the regular part of the free boundary sufficiently far from the origin. This setting corresponds to paraboloid global solutions. In dimension $n=2$, the relevant cross-sections are line segments, so the statement gives no additional geometric information. For the definition of an $\eps$-$C^2$ normal graph, see Definition \ref{def:closeness}.

\begin{thm}\label{theorem:main}
For $n\geq 3$, let $u$ be a perturbed global solution to the obstacle problem in the sense of Definition~\ref{def:sol}, and suppose that $E' \subset \R^{n-1}$ is the ellipsoid associated to $p(x')$. Then for every $\eps>0$ there exists $R_\eps>0$ such that the following statement holds. For $t\in \R$, set
\begin{equation}
    \Sigma_t :=
    \left\{
        x'\in \R^{n-1} : (x',t)\in \{u=0\}
    \right\}.
\end{equation}
Let $x_0=(x_0',t)\in \operatorname{Reg}(u)$ satisfy $t\geq R_{\eps}$. Then $\p\Sigma_t$ is an $\eps$-$C^2$ normal graph over a homothetic copy of $\p E'$. 
\end{thm}

The proof of Theorem \ref{theorem:main} is a modification of the local analysis carried out in \cite{EBERLE2023873}. The argument proceeds in two steps. First, we apply Theorem \ref{thm:acf} to classify blow-downs along moving centers; this is the content of Lemma \ref{lem:classification}. We then show that there exists a sequence of rescalings such that the blow-down limit has a paraboloid or cylinder over an ellipsoid as a coincidence set; this is Proposition \ref{prop:main}. The convergence of regular free boundaries then yields Theorem \ref{theorem:main}.

We note that Theorem \ref{thm:acf} allows us to classify blow-downs along moving centers for perturbed solutions whose blow-down is any homogeneous quadratic polynomial $p$, not just in the case where $\dim(\ker(p))=1$. In particular, neither Lemma \ref{lem:classification} nor Proposition \ref{prop:main} use the assumption  $\dim(\ker(p))=1$. Thus, the remaining cases could be treated by adapting the arguments of \cite{EBERLE2023873}. Since our main purpose is to demonstrate that Theorem \ref{thm:acf} can handle perturbations at infinity, we restrict our attention, for readability, to the case $\dim(\ker(p))=1$. For the corresponding results in the cases $\dim(\ker(p))\geq2$, we refer the reader to the proof of \cite[Main Theorem]{EBERLE2023873}.

\subsection{Structure of the Paper}
In Section \ref{section:1} we will collect some known results on the obstacle problem which carry directly over to the exterior domain setting. In Section \ref{section:2} we  prove Theorem \ref{thm:acf}. Finally, in Section \ref{section:3} we give the proof of Theorem \ref{theorem:main}. In the  appendix, we give a construction of solutions to the exterior domain problem with prescribed blow-down as in Definition \ref{def:sol} as well as an improved Poincar\'e inequality in two dimensions. 

\section{Preliminaries}\label{section:1}
\subsection{Notation}
We work in $\R^n$ where $x = (x',x_n)$. Given any  function $w$, we will denote by $w_{r,x_0}(x) = r^{-2}w(rx+x_0)$ and if $x_0 =0$ we will simply write $w_r$. Moreover, we will write $w_+ := \max\{w,0\}$ and $w_- := \max\{-w,0\}$.
For ease of notation, we write
\begin{equation}
    I^{\pm}(v,x_0, r) = \int_{B_r(x_0)} \frac{\abs{\nabla v_{\pm}}^2}{\abs{x-x_0}^{n-2}}.
\end{equation}
A constant will be called universal if it depends only on the dimension.

\subsection{Solutions to the perturbed problem}
Theorem \ref{theorem:main} concerns  the following class of solutions to \eqref{eq:main}. 
\begin{definition}
\label{def:sol}
    Given $g \in C^{\infty}(\p B_1)$ satisfying $g \geq 0$ we consider $u$ such that
    \begin{itemize}
        \item $\Delta u = \chi_{\{u >0\}}$ in $\R^n \backslash B_1$;
        \item $u = g$ on $\p B_1$;
        \item $u \geq 0$ in $\R^n \backslash B_1$; and
        \item $u_R \to p(x')$ in $L^{\infty}_{\loc}(\R^n\setminus \{0\})$ as $R \to \infty$, where $p$ is some homogeneous quadratic polynomial on $\R^{n-1}$ satisfying $\Delta p =1$ and $p(x') \geq c_p\abs{x'}^2$ for all $x' \in \R^{n-1}$.
    \end{itemize}
\end{definition}

We begin with some regularity results of $u$ away from $\p B_1$.
\begin{lem}
    There exists some universal constant $C$ such that
    \begin{equation}\label{eqn:u:Du:bound}
       \norm{u}_{C^{1,1}(B_6 \backslash B_2)} \leq C\left(\norm{u}_{L^{\infty}(B_7\backslash B_1)} + 1\right).
    \end{equation}
\end{lem}
\begin{proof}
    One can apply the known optimal regularity of solutions to the classical obstacle problem (see for instance \cite[Theorem 2.3]{petrosyan2012regularity}) in $B_{1}(x_0)$ for any $x_0 \in B_6 \backslash B_2$.  
\end{proof}

We will work with the following cut-off of solutions. First extend $u$ to all of $\R^n$ by setting $u=0$ in $B_1$. Now let $\varphi \in C^{\infty}(\R^n)$ such that $\varphi \geq 0$, $\varphi \equiv 0$ on $B_{2}$, $\varphi \equiv 1$ on $\R^n \backslash B_{4}$ and
\begin{equation}
	\label{eqn:bounds:phi}
	\abs{D\varphi} + \abs{D^2\varphi} \leq C.
\end{equation}

We now set $w = u \varphi$ defined on $\R^n$ which satisfies
\begin{equation}
	\label{eqn:laplacian:w}
	\Delta w = \varphi \Delta u + 2 \nabla \varphi\cdot \nabla u + u \Delta \varphi =: f,
\end{equation}
where $f \in L^{\infty}(\R^n)$ with $\text{supp}(f - \chi_{\{w>0\}}) \subset B_{4}$. 
The fact that $f \in L^{\infty}$ follows directly from the bounds \eqref{eqn:u:Du:bound} and \eqref{eqn:bounds:phi}.

\begin{lem}
    There exists some constant $C=C(\norm{f}_{L^{\infty}},n)$ such that
    \begin{equation}\label{eqn:D^2w:bound}
        \norm{D^2w}_{L^{\infty}(\R^n)} \leq C.
    \end{equation}
\end{lem}
\begin{proof}
    We first observe that by the Harnack inequality we have for any $y_0 \in \R^n$ and $R$ such that $\inf_{B_R(y_0)} w =0$ that
    \begin{equation}
       \sup_{B_R(y_0)} w \leq C \inf_{B_R(y_0)} w + C\norm{f}_{L^{\infty}} R^2.
    \end{equation}
    Since $\inf_{B_R(y_0)} w =0$ we conclude that
    \begin{equation}\label{eqn:quad:growth}
        \norm{w}_{L^{\infty}(B_R(y_0))} \leq C R^2.
    \end{equation}
    Now take $x_0 \in \{w>0\} \setminus B_5$ and let $r=\min\{\dist(x_0, \p B_4), \dist(x_0, \p\{w>0\})\}$ then in $B_r(x_0)$ we have $\Delta w = 1$ so that
    \begin{equation}
        \abs{D^2w(x_0)} \leq Cr^{-2}\left( \norm{w}_{L^{\infty}(B_r(x_0))} + \frac{r^2}{2n} \right).
    \end{equation} 
    Now if $r = \dist(x_0, \p\{w>0\})$ then there is some $y_0 \in \p B_r(x_0)$ such that $w(y_0)=0$ and by the quadratic growth \eqref{eqn:quad:growth} we obtain the result. On the other hand if $r = \dist(x_0, \p B_4)$ then $r\geq 1$ and we have that $0 \in B_{4r}(x_0)$ and so again by \eqref{eqn:quad:growth} we obtain the result. Finally, if $x_0 \in B_5$ the bound \eqref{eqn:D^2w:bound} follows immediately from \eqref{eqn:u:Du:bound} and \eqref{eqn:bounds:phi}.
\end{proof}

A consequence of \eqref{eqn:D^2w:bound} and the blow-down assumption in Definition \ref{def:sol} is that
\begin{equation}
    w_R \to p
    \quad\text{in } C^{1,\alpha}_{\loc}(\mathbb R^n),
\end{equation}
as $R \to \infty$ for each $\alpha <1$.

We now record some properties of limits of rescalings of $w$.
\begin{prop}\label{prop:convergence}
    Let $(x^k)_{k\in \N} \subset \p\{w>0\}$ be a sequence of free boundary points such that $\abs{x^k} \to \infty$ and suppose that $r_k \in (0,\infty)$. Then, up to passing to a subsequence, we have 
    \begin{equation}\label{eqn:convergence}
        w_{r_k,x^k} \to w_0 \text{ in } C^{1,\alpha}_{\loc}\cap W^{2,q}_{\loc}(\R^n)
    \end{equation}
    as $k \to \infty$ for each $\alpha <1$ and $q>1$, where $w_0$ is a global solution of the obstacle problem. Moreover, we also have the following convergent of sets
    \begin{equation}\label{eqn:convergence:measure}
        \abs{\left\{ w_{r_k,x^k} = 0 \right\} \cap B_1} \to  \abs{\left\{ w_0 = 0 \right\} \cap B_1}.
    \end{equation}
\end{prop}
\begin{proof}
    It is clear that for each $\alpha <1$ and $p >1$ we have  $w_{r_k,x^k} \to w_0$ in $C^{1,\alpha}_{\loc}(\R^n)$ and weakly in $W^{2,q}_{\loc}(\R^n)$, up to a subsequence, for some $w_0 \in C^{1,1}_{\loc} \cap W^{2,q}_{\loc}(\R^n)$. We will now show that $w_0$ is a global solution of the obstacle problem. Indeed, for any sequence $(r_k)_{k\in\N}$, setting $w_k := w_{r_k,x^k}$ and $f_k(x):= f(r_kx + x^k)$, we have 
    \begin{equation}
        \Delta w_k = \chi_{\{w_k>0\}}\chi_{\mathbb R^n\setminus B_{4/r_k}(-x^k/r_k)} + f_k\chi_{B_{4/r_k}(-x^k/r_k)}.
    \end{equation}
    Since $\abs{x^k} \to \infty$, if $r_k \to r_0 \in [0,
    \infty)$ then we have that for all $\eta \in C^{\infty}_c(\R^n)$ with $\supp(\eta) \subset \{w_0 >0\}$ there is $k$ large enough so that $\supp \eta \cap B_{\frac{4}{r_k}}(-\frac{x^k}{r_k}) = \emptyset$. We therefore find for $k$ large enough that
    \begin{align}
        \int_{\R^n} \eta\Delta w_k = \int_{\R^n} \eta \chi_{B_{\frac{4}{r_k}}(-\frac{x^k}{r_k})^c} + \eta f_k \chi_{B_{\frac{4}{r_k}}(-\frac{x^k}{r_k})} = \int_{\R^n}\eta
    \end{align} 
    and hence by the weak $W^{2,2}_{\loc}$ convergence we conclude that $\Delta w_0 = 1$ in $\{w_0 >0\}$ in the sense of distributions. Since $w_0 \geq 0$ we have that $\nabla w_0 =0$ almost everywhere on $\{w_0=0\}$, hence $D^2 w_0 = 0$ almost everywhere in $\{w_0 =0\}$ and we conclude that $w_0$ is a global solution of the obstacle problem. Now, as a consequence, see \cite[Proposition 3.17]{petrosyan2012regularity}, we obtain the strong $W^{2,q}_{\loc}$ convergence for all $q >1$ and by H\"older we obtain strong $W^{2,1}(B_1)$ convergence. Therefore, for $k$ large enough, we have 
    \begin{align*}
        \abs{\left\{w_k=0\right\} \cap B_1} = \int_{B_1} 1- \Delta w_k \to \int_{B_1} 1- \Delta w_{0} = \abs{\left\{w_0=0\right\} \cap B_1},
    \end{align*}
    as $k \to \infty$ which establishes \eqref{eqn:convergence:measure} in this case. 
    If on the other hand $r_k \to \infty$ then we have that
    \begin{equation}
        \int_{\R^n} \eta f_k \chi_{B_{\frac{4}{r_k}}(-\frac{x^k}{r_k})} \to 0 ,
    \end{equation}
    as $k \to \infty$ and we obtain that $\Delta w_0 = \chi_{\{w_0>0\}}$ as well as the strong $W^{2,q}_{\loc}$ convergence exactly as in the above argument. Finally, we have as $k \to \infty$, that
    \begin{align*}
        \abs{\left\{w_k=0\right\} \cap B_1}
        &= \int_{B_1\backslash B_{4/r_k}(-\frac{x^k}{r_k})} 1- \Delta w_k + \int_{B_1 \cap B_{4/r_k}(-\frac{x^k}{r_k})} \chi_{\{w_k=0\}}\\
        &\to \int_{B_1} 1- \Delta w_{0} \\
        &= \abs{\left\{w_0=0\right\} \cap B_1},
    \end{align*}
    which establishes \eqref{eqn:convergence:measure} in this instance and concludes the proof. 
\end{proof}

We also have the following continuity of measure of the coincidence set with respect to rescalings. 
\begin{lem}
    \label{lem:continuity}
    For each $x_0 \in \R^n$ the function
    \begin{equation}
        r\mapsto \abs{\left\{r^{-2}w(rx+x_0)=0\right\} \cap B_1}
    \end{equation}
    is continuous for $r \in (0,\infty)$. 
\end{lem}
\begin{proof}
The result immediately follows from the continuity of the Lebesgue integral. Indeed, if $r_0 \in (0,\infty)$ and we take $r_k \to r_0$ then we have that
\begin{equation}
    \abs{\left\{r_k^{-2}w(r_kx+x_0)=0\right\} \cap B_1} = r_k^{-n}\int_{B_{r_k}(x_0)} \chi_{\{w=0\}} \to r_0^{-n}\int_{B_{r_0}(x_0)} \chi_{\{w=0\}},
\end{equation}
which proves the statement. 
\end{proof}

We conclude this subsection with the following well known result on the local $C^2$ convergence of regular free boundaries. 
\begin{prop}\label{lem:convergence:boundaries}
    Let $w_k$ and $w_0$ be solutions to the classical obstacle problem in $B_2$ and suppose that
    \begin{equation}
        w_k \to w_0 \text{ in } C^{1,\alpha}_{\loc}(B_2)
    \end{equation}
    for some $\alpha <1$ and $\p \{w_0 >0\} \cap B_2 \subset \operatorname{Reg}(w_0)$. Then $\p\{w_k>0\}\to \p\{w_0>0\}$ in $C^2_{\loc}(B_{3/2})$ in the sense of convergence of graphs.
\end{prop}
\begin{proof}
   See \cite[Footnote 2]{EBERLE2023873}.
\end{proof}

Proposition \ref{lem:convergence:boundaries} motivates the following definition.

\begin{definition}\label{def:closeness}
Let $E' \subset \mathbb R^{n-1}$ be an ellipsoid with $C^2$ boundary, and let $\nu_{E'}$ denote the outward unit normal to $\partial E'$. Given $\eps>0$, we say that $\Gamma \subset \R^{n-1}$ is an $\eps$-$C^2$ normal graph over a homothetic copy of $E'$ if there exist $a \in \mathbb R^{n-1}$, $\lambda>0$, and $\psi \in C^2(U)$ such that
\begin{equation}
    \Gamma = a+\lambda \left\{z+\psi(z)\nu_{E'}(z): z\in E'\right\},
\end{equation}
where
\begin{equation}
    \norm{\psi}_{C^2(E')}\leq \varepsilon.
\end{equation}
Here the $C^2$-norm is computed with respect to the induced metric on $\partial E'$.
\end{definition}

\subsection{Global solutions of the obstacle problem}
We collect here known results on the classification of global solutions to the obstacle problem
\begin{equation}\label{eqn:obstacle}
    \Delta u = \chi_{\{u>0\}}.
\end{equation}

We recall, that for solutions to the classical obstacle problem, blow-down limits with respect to quadratic rescalings were completely classified by Caffarelli in \cite{caffarelli1998obstacle}.
\begin{prop}\label{prop:blowdowns}
    Suppose that $u$ solves \eqref{eqn:obstacle}. Then 
    \begin{equation}
        u_{r}(x) \to u_0
    \end{equation}
    in $C^{1,\alpha}_{\loc}\cap W^{2,q}_{\loc}(\R^n)$ as $r \to \infty$ where $u_0$ is either a homogeneous quadratic polynomial satisfying $\Delta u_0 =1$ or $u_0(x) = \frac{1}{2}(x\cdot e)_+^2$ for some $e \in \mathbb{S}^{n-1}$. A global solution of the form $\frac{1}{2}(x\cdot e)_+^2$ will be called a half-space solution. 
\end{prop}

We recall that $E_{a,z}$ is an $n$-dimensional ellipsoid with center $z \in \R^n$ and semi-axis lengths $a=(a_1,\dots,a_n) \in (0,\infty)^n$ if, after a rotation,
\begin{equation}
    E_{a,z} = \left\{x \in \R^n: \sum_{i=1}^{n} \frac{(x_i-z_i)^2}{a_i^2} \leq 1\right\}.
\end{equation}
We have the following classification result from \cite{di1986bubble}.

\begin{prop}\label{prop:classifcation:compact}
    Suppose that $u$ solves \eqref{eqn:obstacle} and $\{u=0\}$ is compact and has non-empty interior. Then $\{u=0\}$ is an $n$-dimensional ellipsoid. 
\end{prop}

\begin{definition}
   We denote by $E'$ the unique ellipsoid related to $p(x')$ given by $E'=\{v' = 0 \}=E'_{a',z'}$, where $v'$ is the global solution to \eqref{eqn:obstacle} in $\R^{n-1}$ satisfying
    \begin{equation}
        \lim_{\rho \to \infty} \rho^{-2}v'(\rho x') = p(x'),
    \end{equation}
    and such that $z' = 0$ and $a'_1=1$.
\end{definition}

The following is the complete classification obtained in \cite{eberle2022complete}.
\begin{prop}\label{prop:classification}
    If $u$ is a global solution in $\R^n$ with blow-down $p(x')$, $E'$ is the ellipsoid related to $p$, and $\{u=0\}$ has non-empty interior, then $\{u=0\}$ is either a cylinder with $(n-1)$-dimensional ellipsoid $E'$ as its base or the paraboloid
    \begin{equation}
        \left\{(x',x_n) \in \R^n: x' \in \sqrt{x_n} E', x_n \geq 0 \right\} ,
    \end{equation}
    up to homothetic dilations of $E'$ and translations. 
\end{prop}

\subsection{ACF Functional} 
We recall that given two solutions of the obstacle problem $u,v$ that $(u-v)_+$ and $(u-v)_-$ are subaharmonic and so the ACF is monotone for $u-v$. 
\begin{prop}\label{prop:acf:monotone}
    Suppose that $u,v$ are two solutions of the obstacle problem in $B_R(x_0)$. Then $\Phi(u-v,x_0,r)$ is monotone non-decreasing in $r$ for each $r<R$.  
\end{prop}
\begin{proof}
    See for instance \cite[Lemma 2.12 and Lemma 2.13]{eberle2022complete}. 
\end{proof}

In this corollary, there is no restriction on taking $v\equiv 0$.
Applying Proposition \ref{prop:acf:monotone} to difference quotients we obtain the following well known consequence. 
\begin{cor}\label{cor:der:acf}
    Suppose that $u$ is a solution of the obstacle problem in $B_R(x_0)$. Then for each $e \in \p B_1$ we have that $\Phi(\p_e u,x_0,r)$ is monotone non-decreasing in $r$ for each $r<R$.
\end{cor}

We conclude the preliminaries with the following result from \cite[Lemma 14]{caffarelli1998obstacle} which allows us to compare quadratic polynomials using the ACF. 
\begin{lem}\label{lem:acf:quadratic:comparison}
Let $p$ and $q$ be homogeneous quadratic polynomials of the form
\begin{equation}
    \label{eqn:pq:matrix:form}
    p(x)=x^T A x, \qquad q(x)=x^T Q x,
\end{equation}
where $A,Q\in \R^{n\times n}$ are symmetric positive semidefinite matrices such that $\tr(A)=\tr(Q)=\frac12$. Assume that for every $e\in \p B_1$ there holds
\begin{equation}
    \label{eqn:acf:quadratic:comparison}
    \Phi(\p_e q,0,1)\leq \Phi(\p_e p,0,1).
\end{equation}
Then $p\equiv q$.
\end{lem}
\begin{proof}
For each $e\in \p B_1$ we have $\p_e p(x)=2(Ae)\cdot x$ and $\p_e q(x)=2(Qe)\cdot x$. Moreover, if $\ell_a(x)=a\cdot x$, then
\begin{equation}
    \label{eqn:acf:linear:function}
    \Phi(\ell_a,0,1)=c_n\abs{a}^4
\end{equation}
for some dimensional constant $c_n>0$. Hence, from \eqref{eqn:acf:quadratic:comparison}, we obtain that for every $e\in \p B_1$
\begin{equation}
    \label{eqn:matrix:comparison}
    \abs{Qe}^2\leq \abs{Ae}^2.
\end{equation}
Let $N=A-Q$. Then $N$ is symmetric and $\tr N=0$. Since $Q=A-N$, \eqref{eqn:matrix:comparison} gives
\begin{equation}
    \label{eqn:N:comparison}
    \abs{(A-N)e}^2\leq \abs{Ae}^2
\end{equation}
for every $e\in \p B_1$. Choose $e$ to be an eigenvector corresponding to the smallest eigenvalue $\lambda\leq 0$ of $N$. Then \eqref{eqn:N:comparison} gives
\begin{equation}
    \label{eqn:eigenvalue:comparison}
    \abs{Ae-\lambda e}^2\leq \abs{Ae}^2,
\end{equation}
or equivalently
\begin{equation}
    \label{eqn:eigenvalue:expanded}
    -2\lambda e^T A e+\lambda^2\leq 0.
\end{equation}
Since $A$ is non-negative and $\lambda\leq 0$, both terms on the left hand side of \eqref{eqn:eigenvalue:expanded} are non-negative. Hence $\lambda=0$.

Thus the smallest eigenvalue of $N$ is zero. Since $\tr N=0$, all eigenvalues of $N$ are zero, and therefore $N=0$. Hence $A=Q$, and consequently $p\equiv q$.
\end{proof}

\section{Almost monotonicity of the ACF at infinity}\label{section:2}
In this section we will prove Theorem \ref{thm:acf}. Take $w \in C^{1,1}(\R^n)$ solving 
\begin{equation}
    \Delta w = f
\end{equation}
with $f \in L^{\infty}$ satisfying $f=\chi_{\{w>0\}}$ in $\R^n \setminus B_4$. Now suppose that $x_0 \in \p \{w>0\}$ and for some $\tilde{\rho} \geq 1$ consider the solution of 
\begin{equation}
    \begin{cases}
        \Delta v^{\tilde{\rho}} = \chi_{\{v^{\tilde{\rho}}>0\}} & \text{ in } B_{\tilde{\rho}}(x_0) \\
        v^{\tilde{\rho}} = w & \text{ on } \p B_{\tilde{\rho}}(x_0).
    \end{cases}
\end{equation} 

We first show a uniform $W^{1,2}$ estimate on the difference $w-v^{\tilde{\rho}}$.
\begin{lem}\label{lemma:main:difference}
    There exists a constant $C=C(\norm{f}_{L^{\infty}},n)$ such that for each $\tilde{\rho}\geq 4$ we have that
    \begin{equation}
        \label{eqn:w12:diff}
        \norm{\nabla(w-v^{\tilde{\rho}})}_{L^2(B_{\tilde{\rho}}(x_0))} \leq C
    \end{equation}
    in $n\geq 3$ and 
    \begin{equation}
        \label{eqn:w12:diff:2D}
        \norm{\nabla(w-v^{\tilde{\rho}})}_{L^2(B_{\tilde{\rho}}(x_0))} \leq C\sqrt{\log(\tilde{\rho})}
    \end{equation}
    in $n=2$.
\end{lem}
\begin{proof}
    Since $w-v^{\tilde{\rho}} = 0$ on $\p B_{\tilde{\rho}}(x_0)$ and $(w-v^{\tilde{\rho}})(\chi_{\{w>0\}} - \chi_{\{v^{\tilde{\rho}}>0\}}) \geq 0$ outside of $B_4$ we have that
    \begin{align}\label{eqn:first:diff:bound}
        \int_{B_{\tilde{\rho}}(x_0)} \abs{\nabla (w-v^{\tilde{\rho}})}^2 \leq - \int_{B_{4} \cap B_{\tilde{\rho}}(x_0)} (w-v^{\tilde{\rho}}) \Delta (w-v^{\tilde{\rho}}) \leq C\norm{f}_{L^{\infty}} \int_{B_4\cap B_{\tilde{\rho}}(x_0)} \abs{w-v^{\tilde{\rho}}}.
    \end{align} 
    If $n\geq 3$, we obtain by the Sobolev inequality that
    \begin{align}
        \int_{B_4 \cap B_{\tilde{\rho}}(x_0)} \abs{w-v^{\tilde{\rho}}} \leq C\norm{w-v^{\tilde{\rho}}}_{L^{\frac{2n}{n-2}}(B_{\tilde{\rho}}(x_0))} \leq C\norm{\nabla (w-v^{\tilde{\rho}})}_{L^2(B_{\tilde{\rho}}(x_0))},
    \end{align}
    and so combining this with \eqref{eqn:first:diff:bound} and dividing through by $\norm{\nabla (w-v^{\tilde{\rho}})}_{L^2(B_{\tilde{\rho}}(x_0))}$ gives \eqref{eqn:w12:diff}. 

    If $n=2$, then by Lemma~\ref{lem:off:center:poincare:2D} we obtain
    \begin{align}
        \int_{B_4\cap B_{\tilde{\rho}}(x_0)} \abs{w-v^{\tilde{\rho}}} \leq C\norm{w-v^{\tilde{\rho}}}_{L^2(B_4\cap B_{\tilde{\rho}}(x_0))} \leq C\sqrt{\log(8+\tilde{\rho})}\norm{\nabla(w-v^{\tilde{\rho}})}_{L^2(B_{\tilde{\rho}}(x_0))},
    \end{align}
    from which \eqref{eqn:w12:diff:2D} follows.
\end{proof}

We also have the following uniform bound on $D^2v^{\tilde{\rho}}$. 
\begin{lem}
    There is a constant $C=C(\norm{f}_{L^{\infty}},n)$ such that
    \begin{equation}
        \label{eqn:D^2v:bound}
        \norm{D^2v^{\tilde{\rho}}}_{L^{\infty}(B_{\tilde{\rho}/2}(x_0))} \leq C.
    \end{equation}
\end{lem}
\begin{proof}
    This follows from the quadratic growth of solutions to the obstacle problem as long as there is a point $\tilde{x} \in B_{3\tilde{\rho}/4}(x_0)$ such that $v^{\tilde{\rho}}(\tilde{x}) =0$ (see for instance \cite{petrosyan2007geometric}). If on the other hand $\Delta v^{\tilde{\rho}} = 1$ in $B_{3\tilde{\rho}/4}(x_0)$ then for any $\bar{x} \in B_{\tilde{\rho}/2}(x_0)$ we have that $\Delta v^{\tilde{\rho}} =1$ in $B_{\tilde{\rho}/4}(\bar{x})$ and using the derivative estimates for harmonic functions we find
    \begin{align}\label{eqn:D^2v:bound1}
        \abs{D^2v^{\tilde{\rho}}}(\bar{x}) \leq C \tilde{\rho}^{-2}\norm{v^{\tilde{\rho}}}_{L^{\infty}(B_{\tilde{\rho}}(x_0))} + C.
    \end{align}
    Now since $\Delta v^{\tilde{\rho}} \geq 0$ and $v^{\tilde{\rho}}\geq 0$ we have that 
    \begin{equation}\label{eqn:D^2v:bound2}
        \norm{v^{\tilde{\rho}}}_{L^{\infty}(B_{\tilde{\rho}}(x_0))} \leq \norm{w}_{L^{\infty}(\p B_{\tilde{\rho}}(x_0))} \leq C\norm{f}_{L^{\infty}} \tilde{\rho}^2,
    \end{equation}
    where the last inequality follows from the Harnack inequality since $w \geq 0$ and $w(x_0)=0$. Combining \eqref{eqn:D^2v:bound1} and \eqref{eqn:D^2v:bound2} concludes the proof. 
\end{proof}

\begin{lem}\label{lem:acf:diff:estimates}
There exists a constant $C=C(\norm{f}_{L^{\infty}},\norm{D^2w}_{L^{\infty}(\R^n)},n)$ such that for each $h \in (0,1)$ and $e \in \p B_1$ there holds for each $1 \leq r \leq \tilde{\rho}/4$ that
    \begin{equation}
    \label{eqn:quadratic:growth:v}
        I^{\pm}(v^{\tilde{\rho}}_{e,h},x_0, r) \leq Cr^2 , 
    \end{equation}
    \begin{equation}
    \label{eqn:quadratic:growth:w}
         I^{\pm}(w_{e,h},x_0, r) \leq Cr^2  .
    \end{equation}
Moreover, for any $0<\delta<1$ there holds
    \begin{equation}
    \label{eqn:acf:diff}
        \abs{I^{\pm}(v^{\tilde{\rho}}_{e,h},x_0, r) - I^{\pm}(w_{e,h},x_0, r)}
        \leq Cr^2 \left(\frac{\delta^{2-n}}{h}r^{-\frac{1}{2}} + \delta^2\right)
    \end{equation}
    for $n\geq 3$, while for $n=2$ there holds
    \begin{equation}
    \label{eqn:acf:diff:2D}
        \abs{I^{\pm}(v^{\tilde{\rho}}_{e,h},x_0, r) - I^{\pm}(w_{e,h},x_0, r)}
        \leq \frac{C}{h}\left(r\log(\tilde{\rho}) + r^{3/2}\right).
    \end{equation}
\end{lem}
\begin{proof}
The estimate \eqref{eqn:quadratic:growth:v} follows directly from \eqref{eqn:D^2v:bound}. Indeed, for almost every $x \in B_r(x_0)$ we obtain that
\begin{align}\label{eqn:gradientbound:v}
    \abs{\nabla (v^{\tilde{\rho}}_{e,h})_{\pm}(x)} \leq  \norm{D^2 v^{\tilde{\rho}}}_{L^{\infty}(B_{\tilde{\rho}/2}(x_0))}\leq C.
\end{align}
Similarly we obtain for almost every $x \in B_r(x_0)$ that
\begin{equation}\label{eqn:gradientbound:w}
    \abs{\nabla (w_{e,h})_{\pm}(x)} \leq \norm{D^2 w}_{L^{\infty}(B_{\tilde{\rho}/2}(x_0))}
\end{equation}
from which the estimate \eqref{eqn:quadratic:growth:w} follows.

We will now show the  estimate \eqref{eqn:acf:diff} for the $I^+$ terms while  the $I^-$ terms are handled analogously. First observe that
\begin{align}
    I^{+}(v^{\tilde{\rho}}_{e,h},x_0, r) - I^{+}(w_{e,h},x_0, r)
    &= \int_{B_r(x_0)} \frac{\abs{\nabla (v^{\tilde{\rho}}_{e,h})_{+}}^2}{\abs{x-x_0}^{n-2}} - \int_{B_r(x_0)} \frac{\abs{\nabla (w_{e,h})_{+}}^2}{\abs{x-x_0}^{n-2}}\\
    &= \int_{B_{\delta r}(x_0)} \abs{x-x_0}^{2-n}\left(\abs{\nabla (v^{\tilde{\rho}}_{e,h})_{+}}^2 - \abs{\nabla (w_{e,h})_{+}}^2\right) dx \\
    &\hspace{4mm}+ \int_{B_r(x_0)\backslash B_{\delta r}(x_0)} \abs{x-x_0}^{2-n}\left(\abs{\nabla (v^{\tilde{\rho}}_{e,h})_{+}}^2 - \abs{\nabla (w_{e,h})_{+}}^2\right) dx\\
    &=: I + II. 
\end{align}
Now using \eqref{eqn:gradientbound:v} and \eqref{eqn:gradientbound:w} we obtain
\begin{align}
    \abs{I} \leq C\int_{B_{\delta r}(x_0)} \abs{x-x_0}^{2-n} \leq C \int_0^{\delta r} s^{2-n}s^{n-1} ds \leq C\delta^2 r^2. 
\end{align}
For $II$ we will estimate the difference in non-negative parts while the difference in the non-positive parts is handled analogously. To this end we first notice that 
\begin{align}
    \abs{II}
    &\leq \int_{B_r(x_0)\backslash B_{\delta r}(x_0)} \abs{x-x_0}^{2-n}\abs{\abs{\nabla (v^{\tilde{\rho}}_{e,h})}^2 - \abs{\nabla (w_{e,h})}^2} dx\\
    &+ \int_{(B_r(x_0)\backslash B_{\delta r}(x_0))\cap \left\{w_{e,h}\leq 0\right\} \cap \left\{v^{\tilde{\rho}}_{e,h}> 0\right\}} \abs{x-x_0}^{2-n}\abs{\nabla (v^{\tilde{\rho}}_{e,h})_+}^2 dx\\
    &+ \int_{(B_r(x_0)\backslash B_{\delta r}(x_0))\cap \left\{w_{e,h}> 0\right\} \cap \left\{v^{\tilde{\rho}}_{e,h}\leq 0\right\}} \abs{x-x_0}^{2-n}\abs{\nabla (w_{e,h})_+}^2 dx\\
    &\leq (\delta r)^{2-n} \int_{B_r(x_0)}\abs{\abs{\nabla (v^{\tilde{\rho}}_{e,h})}^2 - \abs{\nabla (w_{e,h})}^2} dx\\
    &+ (\delta r)^{2-n}\int_{B_r(x_0)\cap \left\{w_{e,h}\leq 0\right\} \cap \left\{v^{\tilde{\rho}}_{e,h}> 0\right\}} \abs{\nabla v^{\tilde{\rho}}_{e,h}}^2 dx\\
    &+ (\delta r)^{2-n}\int_{B_r(x_0)\cap \left\{w_{e,h}> 0\right\} \cap \left\{v^{\tilde{\rho}}_{e,h}\leq 0\right\}} \abs{\nabla w_{e,h}}^2 dx.
\end{align}
Now using the identity $a^2-b^2=(a-b)(a+b)$ along with the bounds \eqref{eqn:gradientbound:v} and \eqref{eqn:gradientbound:w} we obtain
\begin{align}
    \int_{B_r(x_0)}\abs{\abs{\nabla (v^{\tilde{\rho}}_{e,h})}^2 - \abs{\nabla (w_{e,h})}^2} dx
    &\leq \frac{C}{h}\int_{B_{r+h}(x_0)} \abs{\nabla (v^{\tilde{\rho}}-w)} dx\\
    &\leq \frac{C}{h}r^{n/2} \norm{\nabla (v^{\tilde{\rho}}-w)}_{L^2(B_{\tilde{\rho}}(x_0))}\\
    &\leq\frac{C}{h} r^{n/2} \label{eqn:est1}
\end{align}
where in the last line we have used Lemma \ref{lemma:main:difference}. 
Now we estimate the mismatch sets by observing that for some $\beta<0$ small to be determined we have by \eqref{eqn:gradientbound:v} and the Chebyshev inequality that
\begin{align}
    \int_{B_r(x_0)\cap \left\{w_{e,h}\leq 0\right\} \cap \left\{v^{\tilde{\rho}}_{e,h}> r^{\beta} \right\}} \abs{\nabla v^{\tilde{\rho}}_{e,h}}^2 dx
    &\leq C \abs{B_r(x_0)\cap \left\{w_{e,h}\leq 0\right\} \cap \left\{v^{\tilde{\rho}}_{e,h}> r^{\beta} \right\}}\\
    &\leq C \abs{B_r(x_0) \cap \left\{\abs{v^{\tilde{\rho}}_{e,h} - w_{e,h}} > r^{\beta} \right\}}\\
    &\leq C r^{-2\beta} \norm{v^{\tilde{\rho}}_{e,h} - w_{e,h}}_{L^2(B_r(x_0))}^2\\
    &\leq C r^{-2\beta} \norm{\p_e(v^{\tilde{\rho}} - w)}_{L^2(B_{\tilde{\rho}}(x_0))}^2\\
    &\leq C r^{-2\beta}, \label{eqn:est2}
\end{align}
where in the last line we have used once more Lemma \ref{lemma:main:difference}. Finally, on the set $B_r(x_0)\cap \left\{w_{e,h}\leq 0\right\} \cap \left\{0 < v^{\tilde{\rho}}_{e,h}< r^{\beta} \right\}$ we estimate using the Caccioppoli inequality. Indeed, taking $\eta \in C^{\infty}_c(B_{r+1}(x_0))$ with $\eta \equiv 1$ on $B_r(x_0)$, $\abs{\nabla \eta}\leq C$, and $H(t) = \min\{t_+,r^{\beta}\}$, we have that
\begin{align}
    \int \eta^2 H(v^{\tilde{\rho}}_{e,h}) \Delta v^{\tilde{\rho}}_{e,h} 
    &= - 2 \int  \eta H(v^{\tilde{\rho}}_{e,h}) \nabla v^{\tilde{\rho}}_{e,h} \cdot \nabla \eta - \int \eta^2 H'(v^{\tilde{\rho}}_{e,h}) \abs{\nabla v^{\tilde{\rho}}_{e,h}}^2\\
    &= - 2 \int  \eta H(v^{\tilde{\rho}}_{e,h}) \nabla v^{\tilde{\rho}}_{e,h} \cdot \nabla \eta - \int_{\left\{0 < v^{\tilde{\rho}}_{e,h}< r^{\beta} \right\}} \eta^2 \abs{\nabla v^{\tilde{\rho}}_{e,h}}^2 
\end{align}
and so rearranging this and using the fact that $\abs{\Delta v^{\tilde{\rho}}_{e,h}}\leq \frac{C}{h}$ and the bound \eqref{eqn:gradientbound:v} we find that
\begin{align}
    \int_{B_r(x_0)\cap \left\{w_{e,h}\leq 0\right\} \cap \left\{0<v^{\tilde{\rho}}_{e,h}< r^{\beta} \right\}} \abs{\nabla v^{\tilde{\rho}}_{e,h}}^2 dx
    &\leq \int_{B_r(x_0) \cap \left\{0<v^{\tilde{\rho}}_{e,h}< r^{\beta} \right\}} \abs{\nabla v^{\tilde{\rho}}_{e,h}}^2 dx \\
    &\leq \int_{\left\{0 < v^{\tilde{\rho}}_{e,h}< r^{\beta} \right\}} \eta^2 \abs{\nabla v^{\tilde{\rho}}_{e,h}}^2 \\
    &= - \int \eta^2 H(v^{\tilde{\rho}}_{e,h}) \Delta v^{\tilde{\rho}}_{e,h} - 2 \int \eta H(v^{\tilde{\rho}}_{e,h}) \nabla v^{\tilde{\rho}}_{e,h} \cdot \nabla \eta \\
    &\leq \frac{C}{h}r^{\beta} r^n + Cr^{\beta}r^{n-1} \\
    &\leq \frac{C}{h}r^{\beta} r^n \label{eqn:est3},
\end{align} 
where in the last inequality we have used $r\geq 1$ and $h<1$. Therefore, we find by \eqref{eqn:est2} and \eqref{eqn:est3} that
\begin{equation}
    \int_{B_r(x_0)\cap \left\{w_{e,h}\leq 0\right\} \cap \left\{v^{\tilde{\rho}}_{e,h}>0 \right\}} \abs{\nabla v^{\tilde{\rho}}_{e,h}}^2 dx \leq \frac{C}{h}r^{\beta} r^n + Cr^{-2\beta}.
\end{equation}
An identical argument gives 
\begin{equation}
    \int_{B_r(x_0)\cap \left\{w_{e,h}> 0\right\} \cap \left\{v^{\tilde{\rho}}_{e,h}\leq 0 \right\}} \abs{\nabla w_{e,h}}^2 dx \leq \frac{C}{h}r^{\beta} r^n + Cr^{-2\beta},
\end{equation}
and so we have for $\beta = -1/2$ that
\begin{align}
    \abs{II}\leq (\delta r)^{2-n} \left(\frac{C}{h}r^{\beta} r^n + Cr^{-2\beta} + \frac{C}{h} r^{n/2}\right) \leq  \frac{C}{h} r^2 \delta^{2-n} r^{-1/2}.
\end{align}
Combining this with the estimate for $I$ gives \eqref{eqn:acf:diff}. In $n=2$ we notice that
\begin{align}
    \abs{I^{+}(v^{\tilde{\rho}}_{e,h},x_0, r) - I^{+}(w_{e,h},x_0, r)}
    &\leq \int_{B_r(x_0)}\abs{\abs{\nabla (v^{\tilde{\rho}}_{e,h})}^2 - \abs{\nabla (w_{e,h})}^2} dx\\
    &+ \int_{B_r(x_0)\cap \left\{w_{e,h}\leq 0\right\} \cap \left\{v^{\tilde{\rho}}_{e,h}> 0\right\}} \abs{\nabla v^{\tilde{\rho}}_{e,h}}^2 dx\\
    &+ \int_{B_r(x_0)\cap \left\{w_{e,h}> 0\right\} \cap \left\{v^{\tilde{\rho}}_{e,h}\leq 0\right\}} \abs{\nabla w_{e,h}}^2 dx.
\end{align}
and so the same argument above using \eqref{eqn:w12:diff:2D} instead of \eqref{eqn:w12:diff} to estimate these three terms gives \eqref{eqn:acf:diff:2D} which completes the proof. 
\end{proof}

We can now give:
\begin{proof}[Proof of Theorem \ref{thm:acf}]
    We take $\tilde{\rho}=4R$. In $n\geq 3$ we have that
    \begin{align}
        \Phi(w_{e,h}, x_0, r)
        &= \Phi(v^{\tilde{\rho}}_{e,h},x_0, r) 
        +r^{-4} I^+(v^{\tilde{\rho}}_{e,h},x_0, r) 
        \left(I^-(w_{e,h},x_0, r) - I^-(v^{\tilde{\rho}}_{e,h},x_0, r)\right) \\
        &\hspace{4mm} +r^{-4} I^-(v^{\tilde{\rho}}_{e,h},x_0, r) 
        \left(I^+(w_{e,h},x_0, r) - I^+(v^{\tilde{\rho}}_{e,h},x_0, r)\right)\\
        &\hspace{4mm} +r^{-4}
        \left(I^+(w_{e,h},x_0, r) - I^+(v^{\tilde{\rho}}_{e,h},x_0, r)\right)
        \left(I^-(w_{e,h},x_0, r) - I^-(v^{\tilde{\rho}}_{e,h},x_0, r)\right)\\
        &\leq \Phi(v^{\tilde{\rho}}_{e,h}, x_0, r) + C(\delta,h)r^{-1/2} + \delta^2
        \leq \Phi(v^{\tilde{\rho}}_{e,h}, x_0, R) + C(\delta,h)r^{-1/2} + \delta^2,
    \end{align} 
    where in the last line we have used Proposition \ref{prop:acf:monotone} applied to $v^{\tilde{\rho}}_{e,h} = h^{-1}(v^{\tilde{\rho}}(x+he)-v^{\tilde{\rho}}(x))$. Now, by Lemma \ref{lem:acf:diff:estimates} again we have that
    \begin{equation}
        \Phi(v^{\tilde{\rho}}_{e,h}, x_0, R) \leq  \Phi(w_{e,h}, x_0, R) + C(\delta,h)R^{-1/2} + \delta^2
    \end{equation}
    and since $r \leq R$ we conclude that
    \begin{equation}
        \Phi(w_{e,h}, x_0, r) \leq  \Phi(w_{e,h}, x_0, R) + C(\delta,h)r^{-1/2} + \delta^2,
    \end{equation}
    which concludes the proof for $n\geq 3$ up to redefining $\delta$. In $n=2$, replicating the same argument but this time using \eqref{eqn:acf:diff:2D}, as well as the identity $AB-ab=(A-a)B+a(B-b)$ to bound the term 
    \begin{equation}
        r^{-4} \left(I^+(w_{e,h},x_0, r) - I^+(v^{\tilde{\rho}}_{e,h},x_0, r)\right) \left(I^-(w_{e,h},x_0, r) - I^-(v^{\tilde{\rho}}_{e,h},x_0, r)\right),
    \end{equation} 
    we obtain 
    \begin{equation}
        \Phi(w_{e,h}, x_0, r) \leq  \Phi(w_{e,h}, x_0, R) + C(h)\log(4R)r^{-1/2},
    \end{equation}
    which completes the proof.
\end{proof}

\section{Proof of Theorem \ref{theorem:main}}\label{section:3}
We begin this section with the following consequence of Theorem \ref{thm:acf}.

\begin{lem}\label{lem:classification}
Let $(x^k)_{k\in \N} \subset \p\{w>0\}$ be a sequence of free boundary points such that $\abs{x^k} \to \infty$ and suppose that $r_k \in (0,\infty)$. Then, up to passing to a subsequence, we have that
\begin{equation}\label{eqn:convergence:any}
    w_{r_k,x^k} := r_k^{-2} w(x^k + r_k \cdot) \to w_0 \text{ in } C^{1,\alpha}_{\loc}\cap W^{2,q}_{\loc}(\R^n)
\end{equation}
and either $w_0 = \frac12(x\cdot e)_+^2$ for some $e \in \p B_1$ or
\begin{equation}\label{eqn:convergence:00}
    \rho^{-2} w_0(\rho\cdot) \to p \text{ in } C^{1,\alpha}_{\loc}\cap W^{2,q}_{\loc}(\R^n),
\end{equation}
as $\rho \to \infty$. Moreover, if $\abs{\{w_0 = 0\}}=0$, then $w_0 = p$. 
\end{lem}
\begin{proof}
The convergence in \eqref{eqn:convergence:any} follows directly from Proposition \ref{prop:convergence} while the convergence
\begin{equation}
    \rho^{-2} w_0(\rho\cdot) \to w_{00} \text{ in } C^{1,\alpha}_{\loc}\cap W^{2,q}_{\loc}(\R^n)
\end{equation}
as $\rho \to \infty$ where $w_{00}$ is either a homogeneous quadratic polynomial or a half-space solution follows from Proposition \ref{prop:blowdowns}. 
We will now show that if $w_{00}$ is not a half-space solution then $w_{00}=p$. After passing to a subsequence, either $r_k\geq c_0>0$ for all $k\in\N$ (for some $c_0$ depending on the sequence), or $r_k\to0$ as $k\to\infty$. \\

\noindent\textit{Case 1.}
Fix $\delta >0$ and $h>0$ small.

Observe that since $r_k \geq c_0$ we have that, passing if necessary to a subsequence, $h/r_k \to lh$ for some $l \in [0, c_0^{-1}]$. We therefore have
\begin{equation}\label{eqn:De}
    (w_{r_k,x^k})_{e,h/r_k}  \to   D^{lh}_{e}w_0 :=
    \begin{cases}
        \p_e w_0 & l =0 \\
        (w_0)_{e,lh} & l \neq 0,
    \end{cases}
\end{equation}
in $W^{1,2}_{\loc}$ as $k \to \infty$. Consequently, for each $\rho \geq 1$ and $k$ large enough it holds that
\begin{equation}
    \Phi(D^{lh}_{e}w_0, 0, \rho) \leq \Phi((w_{r_k,x^k})_{e,h/r_k}, 0, \rho) + \delta. 
\end{equation}
Moreover, we have that
\begin{equation}
    (w_{r_k,x^k})_{e,h/r_k}(x) = r_k^{-1}w_{e,h}(x^k+r_kx),
\end{equation}
and hence, by scaling,
\begin{equation}
    \Phi((w_{r_k,x^k})_{e,h/r_k},0,\rho)=\Phi(w_{e,h},x^k,r_k\rho).
\end{equation}

Therefore, for $k$ large enough and $\bar{\rho} = \frac{1}{\eps}\max\{\abs{x^k}, r_k \rho\}$ for some $\eps>0$ small enough depending on $\delta$ (to be determined), we have by Theorem \ref{thm:acf} in $n\geq 3$ and the $W^{2,2}$ convergence, that
\begin{align}
    \Phi(D^{lh}_{e}w_0, 0, \rho)
    &\leq \Phi(w_{e,h}, x^k, r_k\rho) + \delta \\
    &\leq \Phi(w_{e,h}, x^k, \bar{\rho}) + C\delta + C(\delta, h)(\rho r_k)^{-\frac{1}{2}} \\
    &= \Phi((w_{\bar{\rho}})_{e,h/\bar{\rho}}, x^k/\bar{\rho}, 1) + C\delta + C(\delta, h,c_0)\rho^{-\frac{1}{2}} \\
    &\leq \Phi((w_{\bar{\rho}})_{e,h/\bar{\rho}}, 0, 1) + C\delta + C(\delta, h,c_0)\rho^{-\frac{1}{2}}
\end{align} 
where the last inequality follows from \eqref{eqn:D^2w:bound} and dominated convergence for $\eps$ small enough (since $\abs{x^k}/\bar{\rho} \leq \eps$). Now for $\bar{\rho}$ large enough (that is for $k$ large enough) we have by \eqref{eqn:De} (with $l=0$) that
\begin{align}
    \Phi((w_{\bar{\rho}})_{e,h/\bar{\rho}}, 0, 1)\leq \Phi(D^{lh}_e p,0,1) + \delta =\Phi(\p_e p,0,1) + C\delta,
\end{align}
and hence
\begin{align}
    \Phi(D^{lh}_{e}w_0, 0, \rho) \leq \Phi(\p_e p,0,1) + C\delta + C(\delta, h,c_0)\rho^{-\frac{1}{2}}.
\end{align} 
\\
When $n =2$ if $\rho r_k \geq \frac{1}{2}\abs{x^k}$ the above  argument gives 
\begin{align}
    \Phi(D^{lh}_{e}w_0, 0, \rho) 
    &\leq \Phi(\p_e p,0,1) + C\delta + C(h)\log(2\bar{\rho})(r_k\rho)^{-\frac{1}{2}}\\
    &= \Phi(\p_e p,0,1) + C\delta + C(h)\log\left(\frac{1}{\eps}\max\{\abs{x^k}, r_k \rho\}\right)(r_k\rho)^{-\frac{1}{2}}\\
    &\leq \Phi(\p_e p,0,1) + C\delta + C(h,c_0,\eps) \rho^{-1/4}.
\end{align}
If, on the other hand, we have   $\rho r_k \leq \frac{1}{2}\abs{x^k}$, we modify the argument by first applying Proposition \ref{prop:acf:monotone} in $B_{\abs{x^k}/2}(x^k)$ to obtain
\begin{align}
    \Phi(D^{lh}_{e}w_0, 0, \rho)
    &\leq \Phi(w_{e,h}, x^k, r_k\rho) + \delta \\
    &\leq \Phi(w_{e,h}, x^k, \abs{x^k}/2) + \delta \\
    &\leq \Phi((w_{\bar{\rho}})_{e,h/\bar{\rho}}, x^k/\bar{\rho}, 1) + C\delta + C(h)\log(\bar{\rho})\abs{x^k}^{-\frac{1}{2}} \\
    &\leq \Phi((w_{\bar{\rho}})_{e,h/\bar{\rho}}, x^k/\bar{\rho}, 1) + C\delta + C(h)\abs{x^k}^{-\frac{1}{4}} \\
    &\leq \Phi(\p_e p,0,1) + C\delta
\end{align}
for $k$ large enough. In each case we take $\rho \to \infty$ to find that
\begin{equation}
    \Phi(\p_e w_{00},0,1) \leq \Phi(\p_e p,0,1)
\end{equation}
since $\delta >0$ was arbitrary. \\
\medskip
\noindent\textit{Case 2.}
In this case, we notice that given $\rho$ we have for all $k$ large enough that  $\abs{x^k} \to \infty$, and hence
\begin{equation}
    B_{|x_k|/2}(x^k) \cap B_4 = \emptyset
\end{equation}
 Since $r_k \to 0$ we can apply for  $k$ large enough Corollary \ref{cor:der:acf} in $B_{|x_k|/2}(x^k)$ to find
\begin{align}
    \Phi(\p_e w_0, 0, \rho)
    &\leq\Phi(\p_e w, x^k, r_k\rho) + \delta\\
    &\leq \Phi(\p_e w, x^k, \abs{x^k}/2) + \delta \\
    &= \Phi(\p_e w_{\abs{x^k}/2,x^k}, 0, 1) + \delta.
\end{align}
We now observe that 
\begin{equation}
    w_{\abs{x^k}/2,x^k}(x) = \left(\frac{2}{\abs{x^k}}\right)^2w\left(\frac{\abs{x^k}}{2}\left(x+2\frac{x^k}{\abs{x^k}}\right)\right) \to p((x+2\bar{y})')
\end{equation}
as $k \to \infty$ for some $\bar{y} \in \p B_1$ where
\begin{equation}
    \frac{x^k}{\abs{x^k}}\to \bar{y}.
\end{equation}
However, evaluating at $x=0$ we find that $p(\bar{y}') =0$ and so $\bar{y}=\pm e_n$ and hence $p((x+2\bar{y})') = p(x')$. Therefore, for $k$ large enough we have 
\begin{equation}
    \Phi(\p_e w_{0},0,\rho) \leq \Phi(\p_e p,0,1) + C\delta
\end{equation}
for each $e \in \p B_1$. Taking now $\rho \to \infty$ we find
\begin{equation}
    \Phi(\p_{e} w_{00},0,1) \leq \Phi(\p_e p,0,1)
\end{equation}
since $\delta >0$ was arbitrary.\\
Now since $w_{00}$ is a homogeneous quadratic polynomial we have by Lemma \ref{lem:acf:quadratic:comparison} that $w_{00}=p$.\\
The final assertion, namely that $\abs{\{w_0=0\}}=0$ implies $w_0=p$, follows from the fact that if $\abs{\{w_0=0\}}=0$ then $w_0$ is a quadratic polynomial by Liouville's Theorem, and since $w_0(0) = 0$ and $\nabla w_0(0)=0$ we find that $w_0$ is homogeneous and hence $w_{00} = w_{0}$.  
\end{proof}

We require the following Proposition. 
\begin{prop}
\label{prop:main}
Let $x^k \in \p\{w>0\}$ be a sequence of regular free boundary points such that $\abs{x^k} \to \infty$. Then there is a sequence of rescalings $r_k \in (0,2\abs{x^k})$ such that
\begin{equation}
    w_k(x) = r_k^{-2} w(x^k + r_k x) \to w_0 \text{ in } C^{1,\alpha}_{\loc}\cap W^{2,q}_{\loc}(\R^n),
\end{equation}
as $k \to \infty$, where \(\{w_0=0\}\) is, up to translation and homothetic dilation in the \(x'\)-variables, either a paraboloid whose cross-sections are homothetic copies of \(E'\), or a cylinder over a homothetic copy of \(E'\).
\end{prop}
\begin{proof}
    We will first show that there exist two sequences $(\bar{r}_k)_{k\in \N}$ and $(\underline{r}_k)_{k\in \N}$ such that
    \begin{equation}\label{eqn:1/4}
        \abs{\left\{w_{\bar{r}_k,x^k}=0\right\} \cap B_1}  < \frac{1}{4} \abs{B_1} < \abs{\left\{w_{\underline{r}_k,x^k}=0\right\} \cap B_1}.
    \end{equation}
    To this end we fix $\eps > 0$ and take $r_0(\eps)$ large enough such that
    \begin{equation}
        \norm{w_r - p}_{L^{\infty}(B_2)} \leq \eps.
    \end{equation}
    So choosing now for each $k \in \N$ large enough $\overline{r}_k = 2\abs{x^k} \geq r_0$ we have as in the proof of Lemma \ref{lem:classification} that $w_{\bar{r}_k,x^k}(x) \to p(x')$ and so
    \begin{equation}
        \left\{w_{\bar{r}_k,x^k}=0\right\} \cap B_1 \subset \left\{ \abs{x'} \leq C\eps^{1/2} \right\} \cap B_1.
    \end{equation} 
    Since each $x^k$ is a regular free boundary point there exists some $\underline{r}_k$ and a half-space solution $\frac{1}{2}(x\cdot e_k)_+^2$ such that 
    \begin{equation}
        \norm{w_{\underline{r}_k,x^k}-H^k}_{L^{\infty}(B_2)} \leq \eps 
    \end{equation}
    and so by non-degeneracy we have that
    \begin{equation}
        \left\{ x\cdot e_k \leq -C\eps^{1/2} \right\} \cap B_1 \subset \left\{w_{\underline{r}_k,x^k}=0\right\} \cap B_1.
    \end{equation} 
    Hence, for $\eps$ small enough we obtain \eqref{eqn:1/4} and by Lemma \ref{lem:continuity} there exists an $r_k \in (\underline{r}_k, \bar{r}_k)$ such that 
    \begin{equation}
        \abs{\left\{ w_{r_k,x^k} = 0 \right\} \cap B_1} = \frac14 \abs{B_1}.
    \end{equation}
    By Proposition \ref{prop:convergence} we have that $w_{r_k,x^k} \to w_0$ in $W^{1,q}_{\loc}(\R^n) \cap C^{1,\alpha}_{\loc}(\R^n)$ where $w_0$ is a global solution of the obstacle problem satisfying
    \begin{equation}
        \abs{\left\{ w_0 = 0 \right\} \cap B_1} = \frac14 \abs{B_1}.
    \end{equation}
    This implies that $w_0$ cannot be a half-space solution, since then it would have $\abs{\left\{ w_0 = 0 \right\} \cap B_1} = \frac12 \abs{B_1}$. Therefore, by Lemma \ref{lem:classification}, we have that the blow-down of $w_0$ is $p$ and by the classification result Proposition \ref{prop:classification} that $\{w_0=0\}$ is as claimed. 
\end{proof}

\begin{proof}[Proof of Theorem \ref{theorem:main}]
    We argue by contradiction and suppose the conclusion of the Theorem fails. Then there exist $\eps_0>0$ and regular points $x^k=((x^{k})',t_k)$ with $|t_k|\to\infty$ such that $\partial\Sigma_{t_k}$ is not, in any neighborhood of $(x^{k})'$, an $\varepsilon_0$-$C^2$ normal graph over a homothetic copy of $\partial E'$. Now let $r_k$ be the sequence guaranteed by Proposition \ref{prop:main} and redefining $\Sigma_{t_k}$ to be the connected component containing $(x^k)'$ set
    \begin{equation}
        d_k = \diam \left(\Sigma_{t_k}\right).
    \end{equation}
    We observe that 
    \begin{equation}\label{eqn:diam:0}
        \frac{d_k}{\abs{x^k}}\to 0,
    \end{equation}
    and so, in particular, each $d_k$ is finite. Indeed, arguing as in the proof of Lemma \ref{lem:classification} we have that
    \begin{equation}
        \frac{\abs{(y^k)'}}{\abs{y^k}} \to 0
    \end{equation}
    for any sequence $((y^k)',t_k)_{k \in \N} \subset \p \{w>0\}$ from which we see that there exists some sequence $\sigma_k \to 0$ such that $\abs{(y^k)'} \leq 2 \sigma_k \abs{t_k}$. We therefore conclude that 
    \begin{equation}
        \sup_{(y',t_k)\in \partial\Sigma_{t_k}}\frac{\abs{y'}}{\abs{t_k}} \to 0,
    \end{equation}
    from which \eqref{eqn:diam:0} follows by choosing for each $k$, points  $(z^k)',(y^k)' \in \p \Sigma_{t_k}$, such that $ \abs{(z^k)'-(y^k)'} \geq \frac12 d_k$ and noticing that
    \begin{equation}
        \frac{d_k}{\abs{x^k}} \leq 2\frac{\abs{(z^k)'-(y^k)'}}{\abs{x^k}} \leq 2\frac{\abs{(z^k)'}+\abs{(y^k)'}}{\abs{t_k}} \to 0
    \end{equation}
    as $k \to \infty$. 
    Now, after passing to a subsequence, we consider
    \begin{equation}
        w_0(x) = \lim_{k\to\infty} d_k^{-2} w(x^k + d_k x)
    \end{equation}
    and 
    \begin{equation}
        \tilde{w}_0(x) = \lim_{k\to\infty} r_k^{-2} w(x^k + r_k x).
    \end{equation}
    We define the sequence $\lambda_k = d_k/r_k$ and we split the proof into three cases.
    
    \medskip
    \noindent
    \underline{\bf Case 1: $\lambda_k \to \lambda \in (0,\infty)$}\\
    In this case we have that $\tilde{w}_0 = \lambda^{2} w_0(\lambda^{-1} x)$ and so we have that $\{\tilde{w}_0 = 0\}$ is, up to a translation and dilation, either a cylinder over $E'$ or the paraboloid $\gamma \sqrt{x_n} E'$ for some $\gamma >0$. Moreover, since $B_{4/r_k}(-x^k/r_k) \cap B_2 = \emptyset$ for all $k$ large enough (in this case $r_k$ is proportional to $d_k$ so that \eqref{eqn:diam:0} gives $\abs{x^k}/r_k \to \infty$) we can apply Proposition \ref{lem:convergence:boundaries} and obtain that $\Sigma_{t_k}$ converges in $C^2$ to a homothetic copy of $E'$. \\
    
    \medskip
    \noindent
    \underline{\bf Case 2: $\lambda_k \to \infty$}\\
    First observe that $w_0$ is not a half-space solution. Indeed, applying Corollary \ref{cor:der:acf} in $B_{d_k}(x^k)$ (which is disjoint from $B_4$ by \eqref{eqn:diam:0}) we obtain for each $e \in \p B_1$ such that $e \neq \pm e_n$ that
    \begin{equation}
        0 < \frac{1}{2}\Phi(\p_e\tilde{w}_0,0,1) \leq \Phi(\p_ew_0,0,1).
    \end{equation}
    Therefore, by Lemma \ref{lem:classification} we have that the blow-down of $w_0$ is $p$ and so $\{w_0=0\}$ is either a paraboloid $\gamma \sqrt{x_n} E'$, a cylinder over (a homothetic copy) of $E'$ or we have that $w_0 = p$. However, since $p(x') \geq c_p\abs{x'}^2$, this would contradict the fact that we have renormalized the coincidence sets to have diameter bounded from below and so $w_0$ is not $p$. We therefore conclude as in Case 1 above again using the fact that $B_{4/d_k}(-x^k/d_k) \cap B_2 = \emptyset$ so that Proposition \ref{lem:convergence:boundaries} applies.\\

    \medskip
    \noindent
    \underline{\bf Case 3: $\lambda_k \to 0$}\\
    In this case, since $d_k/r_k \to 0$, the limiting coincidence set must have $0$ diameter at the origin and so we have that $\{\tilde{w}_0 =0 \}$ is a paraboloid with tip at the origin. Hence there is some quadratic $f_0(x')$ such that
    \begin{equation}
        \p\{\tilde{w}_0 > 0\} = \left\{(x',x_n) \in \R^n: x_n = f_0(x')\right\}.
    \end{equation}
    Applying  Proposition \ref{lem:convergence:boundaries}  in $B_{1/4}$ (which is disjoint from $B_{4/r_k}(-x^k/r_k)$ since $\abs{x^k} \to \infty$ implies that $\abs{x^k}/r_k>>4/r_k$ for each $r_k$) we obtain for $k$ large enough that
     \begin{equation}
        \p\{r_k^{-2} w(x^k + r_k x) > 0\} = \left\{(x',x_n) \in \R^n: x_n = f_k(x')\right\},
    \end{equation}
    where $f_k \to f_0$ in $C^2$. Let $a_k$ denote the local minimum point of $f_k$ near the origin. Since $f_k\to f_0$ in $C^2$ and $f_0$ has its minimum at the origin, we have $a_k\to 0$. Now defining 
    \[
        \hat f_k(x') := f_k(a_k+x')-f_k(a_k),
    \]
    we have that
    \begin{equation}\label{eqn:case3:need}
        \hat{f}_k(0)=0, \qquad \nabla \hat{f}_k(0)=0. 
    \end{equation}
    Now we transfer this information to the cross sections of $d_k^{-2} w(x^k + d_k x)$ by considering
    \begin{equation}
        g_k(x') = \lambda_k^{-2} \hat{f}_k(\lambda_k x').
    \end{equation}
    Observe by \eqref{eqn:case3:need} we have that $g_k \to f_0$ locally in $C^2$ and for $\tau_k = -\lambda_k^{-2} f_k(a_k)$ there holds
    \begin{equation}
        \diam(\left\{g_k(x') \leq \tau_k\right\}) = 1.
    \end{equation}
    Moreover, since $f_0$ is a quadratic polynomial satisfying $f_0(x') \geq c\abs{x'}^2$, on each fixed ball there exist two constants $T_1, T_2$ independent of $k$ such that $T_1 \abs{x'}^2 \leq g_k(x') \leq T_2 \abs{x'}^2$ and so $\tau_k \in [T_1,T_2]$. Since $g_k \to f_0$ in $C^2$ and, after passing to a subsequence, $\tau_k \to \tau_0 \in [T_1,T_2]$, the implicit function theorem gives us the convergence of the sublevel sets $\{g_k \leq \tau_k\}$ to $\{f_0 \leq \tau_0\}$ which is a rescaling of $E'$. \\

    \medskip
    \noindent
    Hence, in each of the cases above, for $k$ sufficiently large the cross-section $\partial\Sigma_{t_k}$ is an $\eps_0$-$C^2$ normal graph over a homothetic copy of $\partial E'$, contradicting the choice of $x^k$. This completes the proof.
\end{proof}

\appendix
\section{Existence of solutions as in Definition \ref{def:sol}}\label{app:A}
In this short appendix we construct solutions to the exterior domain problem as in Definition \ref{def:sol}.

\begin{lem}[Existence of exterior domain solutions]
Let $K \subset \R^n$ be a compact set with smooth boundary and let
$g : \p K \to [0,\infty)$ be a smooth function. Furthermore we let $p$ be a homogeneous quadratic polynomial on $\R^{n-1}$ satisfying $\Delta p =1$ and $p(x') \geq c \abs{x'}^2$. There exists an exterior domain solution
\begin{equation}
u : \R^n \setminus K \to [0,\infty)
\end{equation}
such that
\begin{equation}
\Delta u = \chi_{\{u>0\}}
\qquad \text{in } \R^n \setminus K,
\end{equation}
\begin{equation}
u=g \qquad \text{on } \p K,
\end{equation}
and we have that
\begin{equation}
\frac{u(Rx)}{R^2} \to p(x')
\qquad \text{in } L^{\infty}_{\loc}(\R^n\setminus \{0\})
\quad \text{as } R\to\infty.
\end{equation}
Furthermore, $\{u=0\}$ is unbounded in the $e_n$-direction, and the
unbounded connected component has nonempty interior.
\end{lem}

\begin{proof}
Let $U$ be a global solution of the obstacle problem with paraboloid
coincidence set $P$ and blow-down $p$. The first-order asymptotics of $U$ imply that
\begin{equation}
\p_n U \leq c <0
\qquad \text{on } \{x_n \le d\}
\end{equation}
for some sufficiently small $d<0$ and so we can translate $U$ such that $U^-(x) = U(x+\tau e_n) \leq g$ on $\p K$ and $U^+(x) = U(x-\tau e_n) \geq g$ on $\p K$ for some $\tau \in (0,\infty)$ large enough. 
Now for every $R$ such that $K\subset B_R$, let $u^R$ be the solution of
\begin{equation}
\begin{cases}
\Delta u^R = \chi_{\{u^R>0\}} & \text{in } B_R\setminus K,\\
u^R = g & \text{on } \p K,\\
u^R = U & \text{on } \p B_R,
\end{cases}
\end{equation}
and we observe that for each $R$ we have
\begin{equation}\label{eqn:sando}
    U^- \leq u^R \leq U^+
\end{equation}
in $B_R \setminus K$ by the comparison principle for the obstacle problem. Hence, along a subsequence $u^R$ converges locally uniformly to a global solution $u$ in $\R^n\setminus K$ satisfying
\begin{equation}
\begin{cases}
    \Delta u = \chi_{\{u>0\}}& \text{in } \R^n \setminus K,\\
u=g &\text{on } \p K.
\end{cases}
\end{equation}
By \eqref{eqn:sando} we have that $\{u=0\}$ contains a paraboloid and so has an unbounded connected component with nonempty interior. Finally, by \eqref{eqn:sando} again we find that 
\begin{equation}
    R^{-2}U^-(Rx) \leq R^{-2}u(Rx) \leq R^{-2}U^+(Rx),
\end{equation}
in $\R^n \setminus \frac{1}{R}K$ and we conclude that
\begin{equation}
\frac{u(Rx)}{R^2} \to p(x')
\qquad \text{in } L^{\infty}_{\loc}(\R^n \setminus \{0\})
\quad \text{as } R\to\infty,
\end{equation}
since $\frac{1}{R} K$ vanishes in the limit.
\end{proof}

\section{A capacitory Poincar\'e inequality in dimension two}
The following Poincar\'e inequality can be found in \cite[Section 14.1.2]{Mazya2011SobolevSpaces}. Since we only require it in a simple setting we give the following elementary proof. 
\begin{lem}[Capacitory Poincar\'e inequality in dimension two]\label{lem:capacitory:poincare:2D}
    Let $R\geq 4$ and suppose that $u\in H^1(B_R)$ with $u=0$ on $\p B_R$ in the sense of traces. Then there exists a universal constant $C$ such that
    \begin{equation}
        \label{eqn:capacitory:poincare:2D}
        \norm{u}_{L^2(B_4)} \leq C\sqrt{\log(R)}\norm{\nabla u}_{L^2(B_R)}.
    \end{equation}
\end{lem}
\begin{proof}
    It is enough to prove the estimate for smooth functions satisfying $u=0$ on $\p B_R$, and then argue by approximation. We write $x=s\theta$, where $s\in (0,R)$ and $\theta\in \p B_1$. Since $u(R\theta)=0$, we have for every $s\in (0,R)$
    \begin{equation}
        \label{eqn:radial:representation}
        u(s\theta)=-\int_s^R \p_\rho u(\rho\theta)d\rho.
    \end{equation}
    Hence, by Hölder's inequality,
    \begin{equation}
        \label{eqn:radial:holder}
        \abs{u(s\theta)}^2\leq \left(\int_s^R \frac{1}{\rho}d\rho\right)\left(\int_s^R \abs{\p_\rho u(\rho\theta)}^2\rho d\rho\right).
    \end{equation}
    Therefore,
    \begin{equation}
        \label{eqn:radial:log:bound}
        \abs{u(s\theta)}^2\leq\log\left(\frac{R}{s}\right)\int_0^R \abs{\nabla u(\rho\theta)}^2\rho d\rho.
    \end{equation}
    Integrating \eqref{eqn:radial:log:bound} over $B_4$ in polar coordinates gives
    \begin{equation}
        \label{eqn:polar:L2:bound}
        \norm{u}_{L^2(B_4)}^2\leq \left(\int_0^4 s\log\left(\frac{R}{s}\right)ds\right) \int_{\p B_1}\int_0^R \abs{\nabla u(\rho\theta)}^2\rho d\rho d\theta.
    \end{equation}
    Since $R\geq 4$,
    \begin{equation}
        \label{eqn:log:integral:bound}
        \int_0^4 s\log\left(\frac{R}{s}\right)ds \leq C\log(R).
    \end{equation}
    Combining \eqref{eqn:polar:L2:bound} and \eqref{eqn:log:integral:bound}, we obtain
    \begin{equation}
        \label{eqn:capacitory:poincare:2D:squared}
        \norm{u}_{L^2(B_4)}^2
        \leq C\log(R)\norm{\nabla u}_{L^2(B_R)}^2.
    \end{equation}
    Taking square roots gives \eqref{eqn:capacitory:poincare:2D}.
\end{proof}

We will need the following off-centered version of the above result. 

\begin{lem}[Off-center capacitory Poincar\'e inequality in dimension two]
\label{lem:off:center:poincare:2D}
Let $R\geq 1$, $x_0\in \mathbb R^2$, and let
$u\in H^1_0(B_R(x_0))$. Then there exists a universal constant $C$ such that
\begin{equation}
    \|u\|_{L^2(B_4\cap B_R(x_0))} \leq C\sqrt{\log(8+R)}\,\|\nabla u\|_{L^2(B_R(x_0))}.
\end{equation}
\end{lem}
\begin{proof}
If $B_4\cap B_R(x_0)=\emptyset$, there is nothing to prove. Otherwise
$|x_0|\leq R+4$. Extend $u$ by zero outside $B_R(x_0)$ and denote the
extension by $\tilde u$. Then
\begin{equation}
    \supp \tilde u \subset B_{|x_0|+R}\subset B_{2R+4}.
\end{equation}
In particular, $\tilde u\in H^1_0(B_{2R+8})$. Applying
Lemma~\ref{lem:capacitory:poincare:2D} to $\tilde u$ in the concentric ball
$B_{2R+8}$ gives
\begin{equation}
    \|\tilde u\|_{L^2(B_4)} \leq C\sqrt{\log(2R+8)}\,\|\nabla \tilde u\|_{L^2(B_{2R+8})}.
\end{equation}
Since
\begin{equation}
    \|\tilde u\|_{L^2(B_4)} \geq \|u\|_{L^2(B_4\cap B_R(x_0))}
\end{equation}
and
\begin{equation}
    \|\nabla \tilde u\|_{L^2(B_{2R+8})}=\|\nabla u\|_{L^2(B_R(x_0))},
\end{equation}
the claim follows.
\end{proof}

\section*{Acknowledgements}
H.S. was supported by the Swedish Research Council (grant no. 2025-03740). A.S. was supported by ERC grant no.~948029. 

\section*{Declarations}

\noindent {\bf  Data availability statement:} All data needed are contained in the manuscript.

\medskip
\noindent {\bf  Funding and/or Conflicts of interests/Competing interests:} The authors declare that there are no financial, competing or conflict of interests.


\bibliographystyle{abbrv}
\bibliography{References.bib}
\end{document}